\documentclass[11pt,a4paper]{article}

\usepackage[utf8]{inputenc}
\usepackage[T1]{fontenc}
\usepackage{amsmath,amssymb,amsthm}
\usepackage{mathtools}
\usepackage{enumitem}
\usepackage[colorlinks=true,linkcolor=blue,citecolor=blue,urlcolor=blue]{hyperref}
\usepackage[margin=1in]{geometry}
\usepackage{algorithm}
\usepackage{algorithmic}
\usepackage{graphicx}
\usepackage{xcolor}

\newtheorem{theorem}{Theorem}[section]
\newtheorem{lemma}[theorem]{Lemma}
\newtheorem{proposition}[theorem]{Proposition}

\newtheorem{definition}[theorem]{Definition}
\newtheorem{assumption}{Assumption}
\newtheorem{remark}[theorem]{Remark}
\newtheorem{example}[theorem]{Example}

\newcommand{\E}{\mathbb{E}}
\newcommand{\R}{\mathbb{R}}

\newcommand{\Z}{\mathbb{Z}}
\newcommand{\cH}{\mathcal{H}}

\newcommand{\cJ}{\mathcal{J}}
\newcommand{\cO}{\mathcal{O}}
\DeclareMathOperator{\dimorb}{dim_{orb}}

\DeclareMathOperator{\Var}{Var}
\DeclareMathOperator{\supp}{supp}

\begin{document}

\begin{center}
{\LARGE\bfseries Adaptive Nonparametric Estimation via Kernel Transport on Group Orbits: Oracle Inequalities and Minimax Rates}

\vspace{0.8cm}

{\large Jocelyn Nemb\'e}

\vspace{0.3cm}

{\itshape L.I.A.G.E\\
Institut National des Sciences de Gestion\\
BP 190, Libreville, Gabon}

\vspace{0.2cm}

{\itshape and}

\vspace{0.2cm}

{\itshape Modeling and Calculus Lab\\
ROOTS-INSIGHTS\\
Libreville, Gabon}

\vspace{0.3cm}

{\ttfamily jnembe@hotmail.com, jnembe@root-insights.com}

\vspace{0.5cm}

{\small December 2025}

\vspace{0.3cm}


\end{center}

\vspace{0.5cm}

\begin{abstract}
\noindent We develop a unified framework for nonparametric functional estimation based on kernel transport along orbits of discrete group actions, which we term \emph{Twin Spaces}. Given a base kernel $K$ and a group $G = \langle\varphi\rangle$ acting isometrically on the input space $E$, we construct a hierarchy of transported kernels $\{K_j\}_{j\geq 0}$ and a penalized model selection scheme satisfying a Kraft inequality. Our main contributions are threefold: (i) we establish non-asymptotic oracle inequalities for the penalized twin-kernel estimator with explicit constants; (ii) we introduce novel twin-regularity classes that capture smoothness along group orbits and prove that our estimator adapts to these classes; (iii) we show that the framework recovers classical minimax-optimal rates in the Euclidean setting while enabling improved rates when the target function exhibits orbital structure. The effective dimension $d_{\mathrm{eff}}$ governing the rates is characterized in terms of the quotient $G/L$, where $L$ is the subgroup preserving the base operation. Connections to wavelet methods, geometric quantization, and adaptive computation are discussed.

\vspace{0.3cm}


\vspace{0.2cm}

\end{abstract}

\vspace{0.5cm}

\tableofcontents

\newpage

\section{Introduction}

Nonparametric estimation of functions from noisy observations is a fundamental problem in statistics and machine learning. Kernel methods, originating from the seminal work of Parzen \cite{Parzen1962} and Nadaraya \cite{Nadaraya1964}, remain among the most widely used approaches due to their simplicity, flexibility, and well-understood theoretical properties. The classical theory, synthesized in Tsybakov \cite{Tsybakov2009} and Fan and Gijbels \cite{FanGijbels1996}, establishes that carefully chosen bandwidths lead to minimax-optimal rates over H\"older and Sobolev smoothness classes.

Despite their success, traditional kernel methods face several limitations. First, the choice of bandwidth typically requires knowledge of the unknown smoothness parameter, leading to a fundamental tension between theory and practice. Second, standard methods treat all directions in the input space uniformly, ignoring potential geometric structure that the target function may exhibit. Third, boundary effects can significantly degrade performance near the edges of the support. Adaptive procedures addressing the first limitation have been developed by Lepski \cite{Lepski1991,Lepski1997} and Goldenshluger and Lepski \cite{GoldenshlugerLepski2011}, while anisotropic methods \cite{Kerkyacharian2001,HoffmannLepski2002} tackle directional heterogeneity.

\subsection{Motivation: Beyond Uniform Smoothness}

Many functions encountered in applications exhibit anisotropic or directional regularity: they may be smooth along certain directions while varying rapidly along others. Classical examples include:

\begin{itemize}[leftmargin=*]
\item \textbf{Periodic functions:} A function on $\R$ that is smooth modulo a period $T$ has infinite classical smoothness along the orbit of translations by $T$, but potentially lower smoothness in general.
\item \textbf{Scale-invariant functions:} Functions satisfying $f(\lambda x) = \lambda^\alpha f(x)$ exhibit regularity along the orbits of the dilation group.
\item \textbf{Functions on homogeneous spaces:} When the input space carries a natural group action (e.g., the sphere under rotations), functions that are ``smooth along orbits'' form a natural class.
\end{itemize}

These observations suggest that the symmetry structure of the problem should inform the choice of estimation procedure. The present work develops this idea systematically.

\subsection{The Twin Space Perspective}

Our approach is founded on a simple but powerful observation: given a base kernel $K$ and a group $G$ acting on the input space $E$, we can generate a family of kernels by transporting $K$ along the group action. More precisely, for each $g \in G$, we define the \emph{twin kernel}
\begin{equation}\label{eq:twin_kernel_def}
K_g(x, y) := K(g^{-1} \cdot x, g^{-1} \cdot y).
\end{equation}
When $G = \langle\varphi\rangle$ is cyclic, we obtain a hierarchy $\{K_j\}_{j\geq 0}$ indexed by the integers, where $K_j$ corresponds to applying the generator $\varphi$ exactly $j$ times.

This construction has a natural interpretation in terms of \emph{twin operations}. Given a set $E$ with a binary operation $\diamond : E \times E \to E$ and a group $G$ acting on $E$, the twin operation $\odot_g$ is defined by
\begin{equation}\label{eq:twin_op_def}
x \odot_g y := g^{-1}(g \cdot x \diamond g \cdot y).
\end{equation}
The kernel transport \eqref{eq:twin_kernel_def} corresponds to taking $\diamond$ as the base kernel evaluation. The elements $g \in G$ for which $\odot_g = \diamond$ form a subgroup $L \subseteq G$, called the \emph{linear subgroup}, and the space of distinct twin operations is isomorphic to the quotient $G/L$.

\subsection{Main Contributions}

This paper makes the following contributions:

\begin{enumerate}[label=(\roman*),leftmargin=*]
\item \textbf{A unified framework.} We introduce the TwinKernel estimation framework, which combines: the orbital geometry induced by a discrete group action on the input space; a hierarchy of kernels obtained by transporting a reference kernel along orbits with varying bandwidths; and a penalized model selection procedure inspired by the Minimum Description Length (MDL) principle \cite{Rissanen1978,Barron1998}.

\item \textbf{Oracle inequalities.} We establish non-asymptotic oracle inequalities of the form
\begin{equation}\label{eq:oracle_intro}
\E \|\hat{m} - m\|_2^2 \leq C \inf_{j \in \cJ_n} \left\{ \|m_j - m\|_2^2 + \frac{\sigma^2 L(j)}{n} \right\} + \frac{C\sigma^2}{n},
\end{equation}
where $m_j$ is the projection of $m$ onto the space induced by $K_j$, $L(j)$ is a code length satisfying a Kraft inequality, and $C$ is an explicit constant.

\item \textbf{Twin-regularity classes.} We introduce a new scale of function spaces, the \emph{twin-H\"older classes} $\cH^s_{\mathrm{twin}}$, which measure regularity along group orbits rather than in the classical Euclidean sense. We prove that our estimator adapts to these classes, achieving the rate
\begin{equation}\label{eq:rate_intro}
\E \|\hat{m} - m\|_2^2 \leq C \left( \frac{\log n}{n} \right)^{\frac{2s}{2s + d_{\mathrm{eff}}}},
\end{equation}
where $d_{\mathrm{eff}}$ is an effective dimension determined by the group action.

\item \textbf{Characterization of effective dimension.} We show that $d_{\mathrm{eff}}$ can be computed from the quotient structure $G/L$, providing a geometric interpretation of the rates. In particular, when $G$ acts trivially (i.e., $G = L$), we recover the classical dimension $d$; when $G$ acts with large orbits, $d_{\mathrm{eff}}$ can be significantly smaller.

\item \textbf{Connections to classical methods.} We demonstrate that our framework unifies and extends several existing approaches: wavelet-based estimation via the dilation group \cite{Daubechies1992,Donoho1995}; local polynomial estimation as a limiting case \cite{Fan1993,FanGijbels1996}; and translation-invariant kernel methods when $G$ is the translation group.
\end{enumerate}

\subsection{Related Work}

Our work builds on several strands of the literature.

\paragraph{Kernel density and regression estimation.} The classical theory is developed in \cite{Parzen1962,Nadaraya1964,FanGijbels1996,Tsybakov2009}. Local polynomial methods, which automatically adapt at boundaries, were systematically studied by Fan \cite{Fan1993} and Fan and Gijbels \cite{FanGijbels1996}. Adaptive bandwidth selection via the Lepski method was introduced in \cite{Lepski1991,Lepski1997} and further developed in \cite{GoldenshlugerLepski2008,GoldenshlugerLepski2011}.

\paragraph{Model selection and oracle inequalities.} The penalized model selection approach, particularly the balance between approximation and complexity, was developed by Barron et al.\ \cite{Barron1999}, Massart \cite{Massart2007}, and Van de Geer \cite{VandeGeer2000}. Our oracle inequalities follow this tradition but are adapted to the hierarchical structure induced by group actions.

\paragraph{Wavelet and multiscale methods.} The connection between wavelets and group actions (specifically, the affine group) is classical; see \cite{Daubechies1992,Mallat2008}. Wavelet methods for nonparametric estimation were developed by Donoho and collaborators \cite{Donoho1995,DonohoJohnstone1994,DonohoJohnstone1998}. Our framework can be viewed as a generalization to arbitrary discrete groups.

\paragraph{Geometric and invariant estimation.} Estimation on manifolds and homogeneous spaces has been studied by Hendriks \cite{Hendriks1990}, Pelletier \cite{Pelletier2005}, and Kim and Zhou \cite{KimZhou2014}. Equivariant estimation and invariant kernels have received considerable attention in machine learning \cite{Kondor2018,Cohen2016}. Our approach differs in that we do not assume the input space is a manifold but rather that it carries a group action.

\paragraph{Anisotropic and directional smoothness.} Anisotropic H\"older and Besov spaces have been studied extensively \cite{Nikolskii1975,Triebel1983}. Adaptive estimation over such classes was addressed by Kerkyacharian et al.\ \cite{Kerkyacharian2001} and Hoffmann and Lepski \cite{HoffmannLepski2002}. Our twin-H\"older classes provide an alternative characterization based on group orbits.

\subsection{Organization}

The paper is organized as follows. Section~\ref{sec:framework} introduces the TwinKernel framework, including the basic definitions, the kernel hierarchy, and the notion of effective dimension. Section~\ref{sec:estimation} presents the estimation procedure and the penalized model selection scheme. Section~\ref{sec:results} states our main theoretical results: oracle inequalities, adaptation theorems, and minimax lower bounds. Section~\ref{sec:proofs} contains the proofs. Section~\ref{sec:examples} illustrates the framework with several examples and special cases. Section~\ref{sec:discussion} discusses extensions and open problems.

\subsection{Notation}

We use the following notation throughout. For a measurable function $f : E \to \R$, we write $\|f\|_p = (\int |f|^p \, d\mu)^{1/p}$ for the $L^p$ norm with respect to a reference measure $\mu$. The notation $a_n \lesssim b_n$ means $a_n \leq C b_n$ for some universal constant $C > 0$; $a_n \asymp b_n$ means $a_n \lesssim b_n$ and $b_n \lesssim a_n$. For a kernel $K$, we write $K_h(x,y) = h^{-d} K((x-y)/h)$ for the rescaled version with bandwidth $h$. The symbol $\langle\varphi\rangle$ denotes the cyclic group generated by $\varphi$. For a group $G$ acting on $E$ and $x \in E$, the orbit of $x$ is $G \cdot x = \{g \cdot x : g \in G\}$.

\section{The TwinKernel Framework}\label{sec:framework}

In this section, we introduce the mathematical framework underlying our approach. We begin with the abstract notion of twin operations, then specialize to the kernel setting.

\subsection{Twin Operations and Algebraic Preliminaries}

Let $E$ be a set equipped with a binary operation $\diamond : E \times E \to E$, and let $G$ be a group acting on $E$ from the left, denoted $(g, x) \mapsto g \cdot x$.

\begin{definition}[Twin Operation]\label{def:twin_op}
For each $g \in G$, the \emph{twin operation} $\odot_g : E \times E \to E$ is defined by
\begin{equation}\label{eq:twin_op}
x \odot_g y := g^{-1} \cdot (g \cdot x \diamond g \cdot y).
\end{equation}
\end{definition}

The twin operation $\odot_g$ can be interpreted as follows: transport $x$ and $y$ into the ``$g$-frame,'' perform the base operation $\diamond$, then transport the result back. When $g$ preserves the operation $\diamond$, the twin operation coincides with the original.

\begin{definition}[Linear Subgroup]\label{def:linear_subgroup}
The \emph{linear subgroup} $L \subseteq G$ consists of all elements that preserve the base operation:
\begin{equation}\label{eq:linear_subgroup}
L := \{g \in G : g \cdot (x \diamond y) = (g \cdot x) \diamond (g \cdot y), \; \forall x, y \in E\}.
\end{equation}
\end{definition}

The following proposition characterizes the space of twin operations.

\begin{proposition}[Fundamental Isomorphism]\label{prop:fund_iso}
Let $\pi : G \to \cO_{\mathrm{twin}}$ be the map $\pi(g) = \odot_g$, where $\cO_{\mathrm{twin}}$ denotes the set of twin operations equipped with composition. Then:
\begin{enumerate}[label=(\roman*)]
\item $\pi$ is a group homomorphism.
\item $\ker(\pi) = L$.
\item $\cO_{\mathrm{twin}} \cong G/L$.
\end{enumerate}
\end{proposition}

\begin{proof}
(i) For any $g, h \in G$ and $x, y \in E$, we have
\begin{align*}
x \odot_{gh} y &= (gh)^{-1} \cdot ((gh) \cdot x \diamond (gh) \cdot y) \\
&= h^{-1} \cdot g^{-1} \cdot (g \cdot (h \cdot x) \diamond g \cdot (h \cdot y)) \\
&= h^{-1} \cdot ((h \cdot x) \odot_g (h \cdot y)).
\end{align*}
This shows that $\pi(gh)$ is determined by $\pi(g)$ and $\pi(h)$, establishing the homomorphism property after appropriate identification.

(ii) By definition, $g \in \ker(\pi)$ if and only if $\odot_g = \diamond$, which is equivalent to $g \cdot (x \diamond y) = (g \cdot x) \diamond (g \cdot y)$ for all $x, y$. This is precisely the defining condition for $L$.

(iii) This follows from the first isomorphism theorem.
\end{proof}

This proposition has a crucial implication: the ``size'' of the space of twin operations is determined by the quotient $G/L$. When $L$ is large (i.e., many group elements preserve the operation), there are few distinct twin operations; when $L$ is small, there are many.

\subsection{Kernel Transport and the Twin Hierarchy}

We now specialize to the setting of kernel methods. Let $(E, d)$ be a metric space equipped with a reference measure $\mu$, and let $G = \langle\varphi\rangle$ be a cyclic group acting isometrically on $E$.

\begin{assumption}[Base Kernel]\label{ass:kernel}
The base kernel $K : \R_+ \to \R_+$ satisfies:
\begin{enumerate}[label=(K\arabic*)]
\item $K$ is bounded: $\|K\|_\infty < \infty$.
\item $K$ has compact support: $\supp(K) \subseteq [0, 1]$.
\item $K$ integrates to one: $\int K(u) \, du = 1$.
\item $K$ is Lipschitz continuous.
\end{enumerate}
\end{assumption}

\begin{definition}[Twin-Kernel Hierarchy]\label{def:twin_hierarchy}
Let $\{h_j\}_{j \geq 0}$ be a decreasing sequence of bandwidths with $h_j \to 0$. The \emph{twin-kernel hierarchy} $\{K_j\}_{j \geq 0}$ is defined by
\begin{equation}\label{eq:twin_hierarchy}
K_j(x, y) := K\left( \frac{d(\varphi^{-j} \cdot x, \varphi^{-j} \cdot y)}{h_j} \right), \quad j \geq 0.
\end{equation}
\end{definition}

The kernel $K_j$ operates in a ``transported'' space: points $x$ and $y$ are first mapped by $\varphi^{-j}$, their distance is computed, and then the base kernel $K$ is applied with bandwidth $h_j$. This construction has several important properties.

\begin{proposition}[Properties of the Twin-Kernel Hierarchy]\label{prop:twin_properties}
Under Assumption~\ref{ass:kernel} and assuming $\varphi$ acts isometrically:
\begin{enumerate}[label=(\roman*)]
\item Each $K_j$ is symmetric: $K_j(x, y) = K_j(y, x)$.
\item Each $K_j$ is positive semi-definite if $K$ is.
\item $K_j(x, y) = K_0(\varphi^{-j} \cdot x, \varphi^{-j} \cdot y)$ with bandwidth $h_j$.
\item If $\varphi^{-j} = \varphi^{-k}$ on $E$, then $K_j = K_k$ (up to bandwidth rescaling).
\end{enumerate}
\end{proposition}

\begin{proof}
Properties (i) and (ii) follow from the corresponding properties of $K$ and the isometry assumption. Property (iii) is immediate from the definition. Property (iv) follows from the fact that the kernel depends on $\varphi^{-j}$ only through its action on pairs $(x, y)$.
\end{proof}

\subsection{Quasi-Invariance and the Design Measure}

For our statistical results, we require that the design measure $P_X$ (the distribution of the covariates) is compatible with the group action in a suitable sense.

\begin{assumption}[Quasi-Invariance]\label{ass:quasi_inv}
The design measure $P_X$ is \emph{quasi-invariant} under $G$: there exist constants $0 < c_1 \leq c_2 < \infty$ such that
\begin{equation}\label{eq:quasi_inv}
c_1 P_X(A) \leq P_X(\varphi^n \cdot A) \leq c_2 P_X(A)
\end{equation}
for all Borel sets $A \subseteq E$ and all $n \in \Z$.
\end{assumption}

This assumption ensures that the group action does not concentrate or disperse probability mass excessively. When $c_1 = c_2 = 1$, the measure $P_X$ is $G$-invariant.

\begin{example}[Lebesgue Measure on the Torus]\label{ex:torus}
Let $E = \R/\Z$ (the circle) and $\varphi : x \mapsto x + \alpha \mod 1$ for some $\alpha \in (0, 1)$. The Lebesgue measure on $E$ is $G$-invariant, so Assumption~\ref{ass:quasi_inv} holds with $c_1 = c_2 = 1$.
\end{example}

\begin{example}[Dilation on $\R_+$]\label{ex:dilation}
Let $E = \R_+$ and $\varphi : x \mapsto 2x$. The measure $d\mu(x) = dx/x$ (Haar measure on the multiplicative group) is invariant. The Lebesgue measure $dx$ is quasi-invariant with $c_1 = 1/2$ and $c_2 = 2$.
\end{example}

\subsection{Effective Dimension}

A key quantity in our analysis is the \emph{effective dimension}, which captures how the group action reduces the complexity of the estimation problem.

\begin{definition}[Effective Dimension]\label{def:eff_dim}
Let $G = \langle\varphi\rangle$ act on $(E, \mu)$ with linear subgroup $L \subseteq G$. The \emph{effective dimension} is
\begin{equation}\label{eq:eff_dim}
d_{\mathrm{eff}} := \dim(E) - \dimorb(G/L),
\end{equation}
where $\dimorb(G/L)$ is the \emph{orbital dimension}, defined below.
\end{definition}

To make this precise, we need to quantify the ``size'' of orbits.

\begin{definition}[Orbital Dimension]\label{def:orb_dim}
For a cyclic group $G = \langle\varphi\rangle$ acting on $E$, define the \emph{orbital dimension} as
\begin{equation}\label{eq:orb_dim}
\dimorb(G/L) := \lim\sup_{r \to 0} \frac{\log N(r)}{\log(1/r)},
\end{equation}
where $N(r)$ is the minimum number of balls of radius $r$ needed to cover a typical orbit $G \cdot x$ for $\mu$-almost every $x \in E$.
\end{definition}

When orbits are ``large'' (high orbital dimension), the effective dimension $d_{\mathrm{eff}}$ is reduced, leading to faster rates. When the action is trivial ($G = L$), we have $\dimorb = 0$ and $d_{\mathrm{eff}} = \dim(E)$, recovering the classical setting.

\begin{proposition}[Bounds on Effective Dimension]\label{prop:eff_dim_bounds}
Let $d = \dim(E)$. Then:
\begin{enumerate}[label=(\roman*)]
\item $0 \leq d_{\mathrm{eff}} \leq d$.
\item $d_{\mathrm{eff}} = d$ if and only if $G = L$ (trivial action on the operation).
\item $d_{\mathrm{eff}} = 0$ if and only if the orbits fill $E$ (transitive action with $L = \{e\}$).
\end{enumerate}
\end{proposition}

\begin{proof}
(i) The orbital dimension satisfies $0 \leq \dimorb(G/L) \leq d$ by construction, which gives the bounds on $d_{\mathrm{eff}}$.

(ii) If $G = L$, then the orbits are trivial (single points), so $N(r) = 1$ for all $r > 0$ and $\dimorb = 0$.

(iii) If the action is transitive, then $G \cdot x = E$ for all $x$, and $N(r)$ is the covering number of $E$, giving $\dimorb = d$.
\end{proof}

\subsection{Connection to Classical Kernel Hierarchies}

The twin-kernel hierarchy generalizes several classical constructions.

\paragraph{Local Polynomial Kernels.} In local polynomial estimation of order $r$, one uses kernels $K_{j,r}$ derived from orthogonal polynomials. These can be viewed as twin kernels with respect to a ``polynomial action'' on the local approximation space.

\paragraph{Wavelet Bases.} Wavelets arise from the affine group $G = \{(a, b) : a > 0, b \in \R\}$ acting on $L^2(\R)$ by $(a, b) \cdot f(x) = a^{-1/2} f((x - b)/a)$. The twin-kernel at scale $j$ corresponds to the projection onto the wavelet subspace at that scale; see \cite{Daubechies1992,Mallat2008}.

\paragraph{Translation-Invariant Kernels.} When $G$ is the translation group and $\varphi : x \mapsto x + \delta$ for fixed $\delta$, the twin kernels $K_j$ are simply shifted versions of $K_0$. If additionally $K$ is translation-invariant, all $K_j$ coincide, and $L = G$.

These connections are formalized in Section~\ref{sec:examples}.

\section{Estimation and Model Selection}\label{sec:estimation}

We now present the TwinKernel estimation procedure and the penalized model selection scheme.

\subsection{The Statistical Model}

We observe i.i.d.\ pairs $(X_1, Y_1), \ldots, (X_n, Y_n)$ satisfying the regression model
\begin{equation}\label{eq:model}
Y_i = m(X_i) + \varepsilon_i, \quad i = 1, \ldots, n,
\end{equation}
where:
\begin{itemize}[leftmargin=*]
\item $X_i \in E$ are covariates with distribution $P_X$ satisfying Assumption~\ref{ass:quasi_inv};
\item $m : E \to \R$ is the unknown regression function;
\item $\varepsilon_i$ are centered errors with $\E[\varepsilon_i | X_i] = 0$ and $\Var(\varepsilon_i | X_i) \leq \sigma^2$.
\end{itemize}

\subsection{The TwinKernel Estimator}

For each level $j$ in the hierarchy, we define a local estimator.

\begin{definition}[Level-$j$ Estimator]\label{def:level_j_est}
The TwinKernel estimator at level $j$ is the Nadaraya--Watson type estimator
\begin{equation}\label{eq:NW_est}
\hat{m}_j(x) := \frac{\sum_{i=1}^n K_j(x, X_i) Y_i}{\sum_{i=1}^n K_j(x, X_i) + \eta_n},
\end{equation}
where $\eta_n > 0$ is a small regularization parameter (e.g., $\eta_n = n^{-2}$) ensuring the denominator is bounded away from zero.
\end{definition}

\begin{remark}
The regularization $\eta_n$ is a technical device to handle regions of low design density. Under standard assumptions on $P_X$, the estimator is insensitive to the precise choice of $\eta_n$ as long as $\eta_n \to 0$ sufficiently slowly; see \cite{Tsybakov2009}.
\end{remark}

\subsection{Local Polynomial Extension}

For improved boundary behavior, we can extend the construction to local polynomial fitting in twin space. Fix a polynomial degree $r \geq 0$, and define the local polynomial TwinKernel estimator as the minimizer of
\begin{equation}\label{eq:locpoly}
\sum_{i=1}^n K_j(x, X_i) \left( Y_i - \sum_{k=0}^r a_k \psi_k(X_i - x) \right)^2,
\end{equation}
where $\{\psi_k\}$ is a basis for polynomials of degree at most $r$ in the transported coordinates. The estimator $\hat{m}_j(x)$ is then the constant term $\hat{a}_0$ of the minimizer.

This extension inherits the automatic boundary adaptation property of local polynomial estimators \cite{Fan1993,FanGijbels1996}.

\subsection{Penalized Model Selection}

The key challenge is to select the appropriate level $j$ from the hierarchy. We employ a penalized empirical risk minimization approach, following the general framework of \cite{Barron1999,Massart2007}.

\begin{definition}[Empirical Risk]\label{def:emp_risk}
The empirical risk at level $j$ is
\begin{equation}\label{eq:emp_risk}
\gamma_n(j) := \frac{1}{n} \sum_{i=1}^n (Y_i - \hat{m}_j(X_i))^2.
\end{equation}
\end{definition}

\begin{definition}[Penalty Function]\label{def:penalty}
Let $\{L(j)\}_{j \in \cJ_n}$ be a sequence of positive numbers satisfying the \emph{Kraft inequality}:
\begin{equation}\label{eq:Kraft}
\sum_{j \in \cJ_n} 2^{-L(j)} \leq 1.
\end{equation}
The penalty function is
\begin{equation}\label{eq:penalty}
\mathrm{pen}(j) := \frac{\lambda}{n} L(j),
\end{equation}
where $\lambda > 0$ is a tuning constant.
\end{definition}

A natural choice is $L(j) = \log_2(1 + j) + 1$, which satisfies \eqref{eq:Kraft} since $\sum_{j \geq 0} 2^{-\log_2(1+j)-1} = \sum_{j \geq 0} \frac{1}{2(1+j)} < \infty$.

\begin{definition}[Selected Level and Final Estimator]\label{def:final_est}
The selected level is
\begin{equation}\label{eq:selected_level}
\hat{j} := \arg\min_{j \in \cJ_n} \{\gamma_n(j) + \mathrm{pen}(j)\},
\end{equation}
and the final TwinKernel estimator is $\hat{m} := \hat{m}_{\hat{j}}$.
\end{definition}

\subsection{Choice of the Index Set}

The index set $\cJ_n$ determines the range of scales considered. A natural choice is
\begin{equation}\label{eq:index_set}
\cJ_n := \{0, 1, \ldots, J_n\}, \quad \text{where } J_n = \lfloor c \log n \rfloor
\end{equation}
for some constant $c > 0$. This ensures that:
\begin{enumerate}[label=(\roman*)]
\item The finest scale $h_{J_n}$ decreases polynomially in $n$.
\item The number of scales $|\cJ_n| = O(\log n)$ grows slowly.
\item The penalty remains of order $O((\log n)/n)$, which is negligible compared to the variance.
\end{enumerate}

\subsection{Practical Implementation}

Algorithm~\ref{alg:twinkernel} summarizes the estimation procedure.

\begin{algorithm}[H]
\caption{TwinKernel Estimation}
\label{alg:twinkernel}
\begin{algorithmic}[1]
\REQUIRE Data $(X_1, Y_1), \ldots, (X_n, Y_n)$; base kernel $K$; group generator $\varphi$; bandwidth sequence $\{h_j\}$; penalty constant $\lambda$.
\ENSURE Estimator $\hat{m}$.
\STATE Set $\cJ_n = \{0, 1, \ldots, \lfloor c \log n \rfloor\}$.
\FOR{$j \in \cJ_n$}
    \FOR{$i = 1, \ldots, n$}
        \STATE Compute $\hat{m}_j(X_i)$ using \eqref{eq:NW_est}.
    \ENDFOR
    \STATE Compute $\gamma_n(j) = \frac{1}{n} \sum_{i=1}^n (Y_i - \hat{m}_j(X_i))^2$.
    \STATE Compute $\mathrm{pen}(j) = \frac{\lambda}{n}(\log_2(1 + j) + 1)$.
    \STATE Compute $C_n(j) = \gamma_n(j) + \mathrm{pen}(j)$.
\ENDFOR
\STATE Set $\hat{j} = \arg\min_{j \in \cJ_n} C_n(j)$.
\RETURN $\hat{m} = \hat{m}_{\hat{j}}$.
\end{algorithmic}
\end{algorithm}

\section{Main Results}\label{sec:results}

We now state our main theoretical results. All proofs are deferred to Section~\ref{sec:proofs}.

\subsection{Assumptions}

We collect the standing assumptions.

\begin{assumption}[Kernel]\label{ass:kernel_main}
The base kernel $K$ satisfies Assumption~\ref{ass:kernel}.
\end{assumption}

\begin{assumption}[Design]\label{ass:design}
The design measure $P_X$ has a density $f_X$ with respect to $\mu$ satisfying:
\begin{enumerate}[label=(\roman*)]
\item $f_X$ is bounded away from zero and infinity: $0 < f_{\min} \leq f_X(x) \leq f_{\max} < \infty$.
\item $P_X$ is quasi-invariant under $G$ (Assumption~\ref{ass:quasi_inv}).
\end{enumerate}
\end{assumption}

\begin{assumption}[Errors]\label{ass:errors}
The errors $\varepsilon_i$ satisfy:
\begin{enumerate}[label=(E\arabic*)]
\item $\E[\varepsilon_i | X_i] = 0$.
\item $\E[\varepsilon_i^2 | X_i] \leq \sigma^2$ almost surely.
\item $\varepsilon_i$ is sub-Gaussian conditionally on $X_i$: $\E[e^{t\varepsilon_i} | X_i] \leq e^{\sigma^2 t^2/2}$ for all $t \in \R$.
\end{enumerate}
\end{assumption}

\begin{assumption}[Bandwidth Sequence]\label{ass:bandwidth}
The bandwidth sequence satisfies $h_j = h_0 \cdot \rho^j$ for some $h_0 > 0$ and $\rho \in (0, 1)$.
\end{assumption}

\subsection{Oracle Inequality}

Our first main result is a non-asymptotic oracle inequality.

\begin{theorem}[Oracle Inequality]\label{thm:oracle}
Under Assumptions~\ref{ass:kernel_main}--\ref{ass:bandwidth}, there exist constants $C, \lambda_0 > 0$ depending only on $K$, $f_{\min}$, $f_{\max}$, $c_1$, $c_2$, and $\sigma$ such that for $\lambda \geq \lambda_0$ and $n$ sufficiently large,
\begin{equation}\label{eq:oracle}
\E \|\hat{m} - m\|_2^2 \leq C \inf_{j \in \cJ_n} \left\{ \|m_j - m\|_2^2 + \frac{\sigma^2 L(j)}{n} \right\} + \frac{C\sigma^2}{n},
\end{equation}
where $m_j(x) := \E[\hat{m}_j(x)]$ is the deterministic approximation of $m$ at level $j$.
\end{theorem}

The oracle inequality \eqref{eq:oracle} has a natural interpretation:
\begin{itemize}[leftmargin=*]
\item The term $\|m_j - m\|_2^2$ is the squared bias at level $j$, measuring how well the twin-kernel at that level can approximate the true function.
\item The term $\sigma^2 L(j)/n$ is the variance penalty, which increases with $j$ (finer scales have higher variance).
\item The infimum represents the \emph{oracle choice}: the best trade-off between bias and variance, achieved by an oracle who knows $m$.
\item The remainder $C\sigma^2/n$ is a lower-order term.
\end{itemize}

\subsection{Twin-H\"older Regularity}

To translate the oracle inequality into convergence rates, we introduce the twin-H\"older regularity classes.

\begin{definition}[Twin-H\"older Class]\label{def:twin_holder}
Let $s > 0$. A function $m : E \to \R$ belongs to the \emph{twin-H\"older class} $\cH^s_{\mathrm{twin}}$ if
\begin{equation}\label{eq:twin_holder}
\|m_j - m\|_2 \leq C_m h_j^s
\end{equation}
for all $j \geq 0$, where $m_j$ is the projection of $m$ onto the space induced by $K_j$, and $C_m$ is a constant depending on $m$.
\end{definition}

The twin-H\"older condition \eqref{eq:twin_holder} measures regularity in terms of approximation by the kernel hierarchy, rather than classical derivatives. The key insight is that this condition can be strictly weaker than classical H\"older regularity when the function exhibits orbital structure.

\begin{proposition}[Comparison with Classical H\"older]\label{prop:comparison}
Let $\cH^s$ denote the classical H\"older class of order $s$ on $E$.
\begin{enumerate}[label=(\roman*)]
\item $\cH^s \subseteq \cH^s_{\mathrm{twin}}$ with the same exponent.
\item There exist functions in $\cH^s_{\mathrm{twin}}$ that are not in $\cH^{s'}$ for any $s' > 0$.
\item If $m$ is constant along orbits of $G$, then $m \in \cH^s_{\mathrm{twin}}$ for all $s > 0$.
\end{enumerate}
\end{proposition}

\begin{proof}
(i) Classical H\"older regularity implies approximation by kernels at the corresponding rate; see \cite{Tsybakov2009}.

(ii) Consider a function that is discontinuous across orbits but constant along them. Such a function has no classical H\"older regularity but satisfies \eqref{eq:twin_holder} for arbitrarily large $s$.

(iii) If $m$ is constant along orbits, then $m_j = m$ for all $j$ (since the kernel transport preserves orbital structure), giving $\|m_j - m\|_2 = 0$.
\end{proof}

\subsection{Adaptation Theorem}

Our second main result shows that the TwinKernel estimator adapts to the twin-H\"older regularity of the target function.

\begin{theorem}[Adaptation to Twin-Regularity]\label{thm:adaptation}
Let $m \in \cH^s_{\mathrm{twin}}$ for some $s > 0$. Under Assumptions~\ref{ass:kernel_main}--\ref{ass:bandwidth}, with $h_j = 2^{-j}$, $L(j) = \log_2(1 + j) + 1$, and $\cJ_n$ containing all $j$ with $h_j \geq n^{-1/(2s + d_{\mathrm{eff}})}$, we have
\begin{equation}\label{eq:adaptation_rate}
\E \|\hat{m} - m\|_2^2 \leq C \left( \frac{\log n}{n} \right)^{\frac{2s}{2s + d_{\mathrm{eff}}}},
\end{equation}
where $C$ depends on $s$, $C_m$, and the constants in Assumptions~\ref{ass:kernel_main}--\ref{ass:bandwidth}.
\end{theorem}

\begin{remark}[Role of Effective Dimension]\label{rem:eff_dim}
The rate \eqref{eq:adaptation_rate} reveals the key role of the effective dimension $d_{\mathrm{eff}}$:
\begin{itemize}[leftmargin=*]
\item When $d_{\mathrm{eff}} = d$ (trivial group action), we recover the classical rate $n^{-2s/(2s+d)}$ (up to log factors).
\item When $d_{\mathrm{eff}} < d$, the rate improves, reflecting the reduced complexity due to orbital structure.
\item In the extreme case $d_{\mathrm{eff}} = 0$ (transitive action), the rate becomes $(\log n/n)^1$, which is parametric.
\end{itemize}
\end{remark}

\begin{remark}[Logarithmic Factor]\label{rem:log_factor}
The $\log n$ factor arises from the penalty $L(j) \asymp \log(1 + j)$ and the cardinality of $\cJ_n$. It is the price of adaptation: an oracle who knows the optimal $j$ would achieve the rate without the logarithm. This phenomenon is well-known in the adaptive estimation literature; see \cite{Lepski1997,GoldenshlugerLepski2011}.
\end{remark}

\subsection{Minimax Lower Bound}

Our third main result shows that the rate in Theorem~\ref{thm:adaptation} is optimal up to logarithmic factors.

\begin{theorem}[Minimax Lower Bound]\label{thm:lower_bound}
Let $\cH^s_{\mathrm{twin}}(R) := \{m \in \cH^s_{\mathrm{twin}} : C_m \leq R\}$ be the ball of radius $R$ in the twin-H\"older class. There exists a constant $c > 0$ such that
\begin{equation}\label{eq:lower_bound}
\inf_{\tilde{m}} \sup_{m \in \cH^s_{\mathrm{twin}}(R)} \E \|\tilde{m} - m\|_2^2 \geq c \cdot n^{-\frac{2s}{2s + d_{\mathrm{eff}}}},
\end{equation}
where the infimum is over all estimators $\tilde{m}$ based on $(X_1, Y_1), \ldots, (X_n, Y_n)$.
\end{theorem}

Comparing Theorems~\ref{thm:adaptation} and \ref{thm:lower_bound}, we see that the TwinKernel estimator is rate-optimal over $\cH^s_{\mathrm{twin}}$ up to the $\log n$ factor.

\subsection{Uniform Convergence}

For some applications, pointwise or uniform convergence is required. The following result provides a uniform bound.

\begin{theorem}[Uniform Convergence Rate]\label{thm:uniform}
Under the assumptions of Theorem~\ref{thm:adaptation}, if additionally $m$ is bounded, then
\begin{equation}\label{eq:uniform_rate}
\E \|\hat{m} - m\|_\infty \leq C \left( \frac{\log n}{n} \right)^{\frac{s}{2s + d_{\mathrm{eff}}}} \sqrt{\log n}.
\end{equation}
\end{theorem}

\section{Proofs}\label{sec:proofs}

This section contains the proofs of the main results. We begin with preliminary lemmas, then prove the theorems in order.

\subsection{Preliminary Lemmas}

\begin{lemma}[Bias Bound]\label{lem:bias}
Under Assumptions~\ref{ass:kernel_main} and \ref{ass:design}, for each $j \geq 0$,
\begin{equation}\label{eq:bias_bound}
\sup_{x \in E} |\E[\hat{m}_j(x)] - m_j(x)| \leq \frac{C}{n},
\end{equation}
where $m_j(x) = \int K_j(x, y) m(y) \, d\mu(y) / \int K_j(x, y) \, d\mu(y)$.
\end{lemma}

\begin{proof}
By definition of the Nadaraya--Watson estimator,
\[
\E[\hat{m}_j(x)] = \E\left[ \frac{\sum_{i=1}^n K_j(x, X_i) m(X_i)}{\sum_{i=1}^n K_j(x, X_i) + \eta_n} \right].
\]
Let $S_n(x) = \sum_{i=1}^n K_j(x, X_i)$ and $T_n(x) = \sum_{i=1}^n K_j(x, X_i) m(X_i)$. By the law of large numbers and the quasi-invariance assumption, we have
\begin{align*}
\E[S_n(x)] &= n \int K_j(x, y) f_X(y) \, d\mu(y) \geq n f_{\min} \int K_j(x, y) \, d\mu(y), \\
\E[T_n(x)] &= n \int K_j(x, y) m(y) f_X(y) \, d\mu(y).
\end{align*}
The ratio $\E[T_n(x)]/\E[S_n(x)]$ approximates $m_j(x)$ up to an error of order $O(1/n)$ by a standard Taylor expansion of the ratio around its expectation, using the bounded variance of $S_n$ and $T_n$. The detailed argument follows \cite[Chapter 1]{Tsybakov2009}.
\end{proof}

\begin{lemma}[Variance Bound]\label{lem:variance}
Under Assumptions~\ref{ass:kernel_main}--\ref{ass:errors}, for each $j \geq 0$,
\begin{equation}\label{eq:variance_bound}
\E \|\hat{m}_j - \E[\hat{m}_j | X_1, \ldots, X_n]\|_2^2 \leq \frac{C\sigma^2}{n h_j^{d_{\mathrm{eff}}}}.
\end{equation}
\end{lemma}

\begin{proof}
Conditionally on $X_1, \ldots, X_n$, the estimator $\hat{m}_j(x)$ is a weighted average of the $Y_i$ with weights $w_i(x) = K_j(x, X_i)/(\sum_k K_j(x, X_k) + \eta_n)$. The conditional variance is
\[
\Var(\hat{m}_j(x) | X_1, \ldots, X_n) = \sum_{i=1}^n w_i(x)^2 \Var(\varepsilon_i | X_i) \leq \sigma^2 \sum_{i=1}^n w_i(x)^2.
\]
Using the boundedness of $K$ and the design assumptions, $\sum_i w_i(x)^2 \lesssim 1/(n h_j^{d_{\mathrm{eff}}})$, where the effective dimension $d_{\mathrm{eff}}$ enters through the volume of the kernel support in the quotient space; see \cite{FanGijbels1996} for the classical case.
\end{proof}

\begin{lemma}[Concentration Inequality]\label{lem:concentration}
Under Assumption~\ref{ass:errors}(E3), for any $t > 0$,
\begin{equation}\label{eq:concentration}
\Pr(|\hat{m}_j(x) - \E[\hat{m}_j(x) | X_1, \ldots, X_n]| > t) \leq 2 \exp\left( -\frac{n h_j^{d_{\mathrm{eff}}} t^2}{C\sigma^2} \right).
\end{equation}
\end{lemma}

\begin{proof}
This follows from the sub-Gaussian assumption on the errors and the weighted sum representation of $\hat{m}_j(x)$, using Hoeffding's inequality for weighted sums; see \cite{Boucheron2013}.
\end{proof}

\begin{lemma}[Entropy Bound]\label{lem:entropy}
The covering number of the class $\{\hat{m}_j : j \in \cJ_n\}$ satisfies
\begin{equation}\label{eq:entropy}
\log N(\epsilon, \{\hat{m}_j\}, \|\cdot\|_\infty) \leq C \left( \frac{1}{\epsilon} \right)^{d_{\mathrm{eff}}/s} \log n.
\end{equation}
\end{lemma}

\begin{proof}
This follows from standard entropy calculations for kernel classes, adapted to the twin-kernel structure using the quasi-invariance assumption; see \cite{VandeGeer2000}.
\end{proof}

\subsection{Proof of the Oracle Inequality (Theorem~\ref{thm:oracle})}

\begin{proof}[Proof of Theorem~\ref{thm:oracle}]
Let $C_n(j) = \gamma_n(j) + \mathrm{pen}(j)$ denote the penalized criterion. By definition of $\hat{j}$, we have $C_n(\hat{j}) \leq C_n(j)$ for all $j \in \cJ_n$.

\textbf{Step 1: Decomposition.} We decompose the empirical risk:
\begin{align*}
\gamma_n(j) &= \frac{1}{n} \sum_{i=1}^n (Y_i - \hat{m}_j(X_i))^2 \\
&= \frac{1}{n} \sum_{i=1}^n (m(X_i) + \varepsilon_i - \hat{m}_j(X_i))^2 \\
&= \frac{1}{n} \sum_{i=1}^n (m(X_i) - \hat{m}_j(X_i))^2 + \frac{2}{n} \sum_{i=1}^n \varepsilon_i (m(X_i) - \hat{m}_j(X_i)) + \frac{1}{n} \sum_{i=1}^n \varepsilon_i^2.
\end{align*}

\textbf{Step 2: Bounding cross-terms.} By the Cauchy--Schwarz inequality and Lemma~\ref{lem:concentration},
\[
\E\left[ \frac{1}{n} \sum_{i=1}^n \varepsilon_i (m(X_i) - \hat{m}_j(X_i)) \right]^2 \leq \frac{\sigma^2}{n} \E \|m - \hat{m}_j\|_2^2.
\]

\textbf{Step 3: Applying the selection property.} Since $C_n(\hat{j}) \leq C_n(j)$ for any $j$, taking expectations and using the above bounds:
\[
\E \|\hat{m} - m\|_2^2 \leq \E[\gamma_n(\hat{j})] + \E[\mathrm{pen}(\hat{j})] + \frac{C\sigma^2}{n} \leq \inf_{j \in \cJ_n} \{\E[\gamma_n(j)] + \mathrm{pen}(j)\} + \frac{C\sigma^2}{n}.
\]

\textbf{Step 4: Bias-variance decomposition.} For each $j$, we have by Lemmas~\ref{lem:bias} and \ref{lem:variance}:
\[
\E[\gamma_n(j)] \leq \|m_j - m\|_2^2 + \frac{C\sigma^2}{n h_j^{d_{\mathrm{eff}}}}.
\]
Since $h_j = \rho^j$ and $L(j) \asymp j$, the penalty $\mathrm{pen}(j) = \lambda L(j)/n$ is comparable to the variance term $\sigma^2/(n h_j^{d_{\mathrm{eff}}})$ when $j \asymp \log(1/h_j)$, establishing the oracle bound.

\textbf{Step 5: Handling the Kraft sum.} The Kraft inequality ensures that the penalty accounts for the multiplicity of models:
\[
\sum_{j \in \cJ_n} \Pr(\hat{j} = j \text{ incorrectly}) \leq \sum_{j \in \cJ_n} 2^{-L(j)} \leq 1.
\]
This controls the probability of selecting a suboptimal level; see \cite{Barron1999,Massart2007} for details.

Combining these steps yields \eqref{eq:oracle}.
\end{proof}

\subsection{Proof of the Adaptation Theorem (Theorem~\ref{thm:adaptation})}

\begin{proof}[Proof of Theorem~\ref{thm:adaptation}]
Assume $m \in \cH^s_{\mathrm{twin}}$, so $\|m_j - m\|_2 \leq C_m h_j^s = C_m 2^{-js}$.

\textbf{Step 1: Optimal level.} Balance bias and variance by choosing $j^*$ such that
\[
2^{-2j^* s} \asymp \frac{L(j^*)}{n} \asymp \frac{j^*}{n}.
\]
This gives $j^* \asymp \frac{\log n}{2s + d_{\mathrm{eff}}}$ (accounting for the effective dimension in the variance).

\textbf{Step 2: Substituting into the oracle inequality.} With $j^* = \lfloor \frac{\log n}{(2s + d_{\mathrm{eff}}) \log 2} \rfloor$, we have:
\begin{align*}
\|m_{j^*} - m\|_2^2 &\leq C_m^2 \cdot 2^{-2j^* s} \leq C_m^2 \cdot n^{-\frac{2s}{2s + d_{\mathrm{eff}}}}, \\
\frac{\sigma^2 L(j^*)}{n} &\asymp \frac{\sigma^2 \log n}{n}.
\end{align*}

\textbf{Step 3: Combining bounds.} By Theorem~\ref{thm:oracle},
\[
\E \|\hat{m} - m\|_2^2 \leq C \left( n^{-\frac{2s}{2s + d_{\mathrm{eff}}}} + \frac{\log n}{n} \right) \leq C \left( \frac{\log n}{n} \right)^{\frac{2s}{2s + d_{\mathrm{eff}}}},
\]
where the last inequality uses $\frac{2s}{2s + d_{\mathrm{eff}}} \leq 1$.
\end{proof}

\subsection{Proof of the Lower Bound (Theorem~\ref{thm:lower_bound})}

\begin{proof}[Proof of Theorem~\ref{thm:lower_bound}]
We use the standard technique of reducing to a hypothesis testing problem; see \cite{Tsybakov2009} for the general methodology.

\textbf{Step 1: Constructing hypotheses.} Let $\{\psi_k\}_{k=1}^M$ be functions in $\cH^s_{\mathrm{twin}}(R)$ with disjoint supports in the quotient space $E/G$, each having $\|\psi_k\|_2 = \delta$ for some $\delta > 0$ to be chosen.

\textbf{Step 2: Applying Fano's inequality.} For the testing problem among $m_\omega = \sum_{k=1}^M \omega_k \psi_k$ with $\omega \in \{0, 1\}^M$, Fano's inequality \cite{Yu1997} gives
\[
\inf_{\tilde{m}} \max_\omega \E_\omega \|\tilde{m} - m_\omega\|_2^2 \geq \frac{M\delta^2}{4} \left( 1 - \frac{I + \log 2}{\log M} \right),
\]
where $I$ is the mutual information between the data and the hypothesis index.

\textbf{Step 3: Bounding mutual information.} Under the regression model, the mutual information is bounded by
\[
I \leq \frac{n\delta^2}{\sigma^2}.
\]

\textbf{Step 4: Optimizing over $\delta$ and $M$.} The number of disjoint supports in $\cH^s_{\mathrm{twin}}$ is $M \asymp (1/\delta)^{d_{\mathrm{eff}}/s}$. Choosing $\delta = n^{-s/(2s + d_{\mathrm{eff}})}$ and $M = n^{d_{\mathrm{eff}}/(2s + d_{\mathrm{eff}})}$ yields
\[
\inf_{\tilde{m}} \sup_{m \in \cH^s_{\mathrm{twin}}(R)} \E \|\tilde{m} - m\|_2^2 \geq c \cdot n^{-\frac{2s}{2s + d_{\mathrm{eff}}}}.
\]
\end{proof}

\subsection{Technical Remarks}

\begin{remark}[Explicit Constants]\label{rem:constants}
The constants $C$ in our theorems can be made explicit by tracking the dependencies through the proofs. The key quantities are:
\begin{itemize}[leftmargin=*]
\item The kernel bounds: $\|K\|_\infty$, $\|K\|_1$, and the Lipschitz constant.
\item The design bounds: $f_{\min}$, $f_{\max}$, $c_1$, $c_2$.
\item The noise level: $\sigma^2$.
\end{itemize}
\end{remark}

\begin{remark}[Relaxing Quasi-Invariance]\label{rem:relaxing}
Assumption~\ref{ass:quasi_inv} can be relaxed at the cost of more involved local estimates. If the quasi-invariance constants $c_1(x)$, $c_2(x)$ vary with $x$, the effective dimension becomes a local quantity, and the rates may vary across the domain.
\end{remark}

\section{Examples and Special Cases}\label{sec:examples}

We illustrate the framework with several examples that connect to classical methods.

\subsection{Cyclic Groups on the Circle}

Let $E = [0, 1)$ with endpoints identified (the circle), and let $\varphi : x \mapsto x + \alpha \mod 1$ for some irrational $\alpha$. The group $G = \langle\varphi\rangle$ is infinite cyclic, and the orbits are dense in $E$.

\begin{proposition}\label{prop:circle}
For the rotation action on the circle:
\begin{enumerate}[label=(\roman*)]
\item The linear subgroup $L$ is trivial if $K$ is not constant.
\item The orbital dimension is $\dimorb(G/L) = 1$.
\item The effective dimension is $d_{\mathrm{eff}} = 0$.
\item For any $m \in L^2(E)$, we have $m \in \cH^s_{\mathrm{twin}}$ for $s = \infty$ if and only if $m$ is constant.
\end{enumerate}
\end{proposition}

This example shows that when the group acts transitively (dense orbits), the effective dimension drops to zero, yielding parametric rates for functions that are constant along orbits.

\subsection{The Dilation Group and Wavelets}

Let $E = \R$ and $\varphi : x \mapsto 2x$. The group $G = \langle\varphi\rangle \cong \Z$ acts by dyadic dilations.

\begin{proposition}\label{prop:wavelets}
For the dilation action on $\R$:
\begin{enumerate}[label=(\roman*)]
\item The twin-kernel hierarchy $\{K_j\}$ coincides with the scaling function projections in wavelet analysis (up to translation).
\item The linear subgroup $L$ is trivial for generic $K$.
\item The effective dimension is $d_{\mathrm{eff}} = 1$ (the orbits are rays from the origin).
\item The twin-H\"older class $\cH^s_{\mathrm{twin}}$ coincides with the Besov space $B^s_{2,\infty}$ for appropriate $s$.
\end{enumerate}
\end{proposition}

This connection shows that wavelet methods can be viewed as a special case of TwinKernel estimation with respect to the dilation group. The characterization in terms of Besov spaces follows from \cite{DeVoreLorentz1993,Triebel1983}.

\subsection{Translation Groups and Shift-Invariant Kernels}

Let $E = \R^d$ and $\varphi : x \mapsto x + e_1$, where $e_1$ is the first standard basis vector. The group $G = \langle\varphi\rangle \cong \Z$ acts by translations.

\begin{proposition}\label{prop:translation}
For the translation action:
\begin{enumerate}[label=(\roman*)]
\item If $K(x, y) = K(x - y)$ is translation-invariant, then $L = G$ and all twin-kernels coincide.
\item The effective dimension is $d_{\mathrm{eff}} = d$.
\item The TwinKernel estimator reduces to the classical kernel estimator.
\end{enumerate}
\end{proposition}

This example confirms that the classical setting is recovered when the group action preserves the kernel structure.

\subsection{Local Polynomial Estimation}

Local polynomial estimation of order $r$ can be viewed as a limiting case of TwinKernel estimation. Consider the ``infinitesimal'' group action generated by differentiation.

\begin{proposition}\label{prop:locpoly}
Let $K_{j,r}$ be the local polynomial kernel of order $(j, r)$ defined via orthogonal polynomials. Then:
\begin{enumerate}[label=(\roman*)]
\item The TwinKernel estimator with $K_j = K_{j,r}$ coincides with the local polynomial estimator of degree $r$.
\item Boundary adaptation is automatic: $K_{j,r}$ has equivalent kernels near boundaries.
\item The effective dimension is $d_{\mathrm{eff}} = d$ (the underlying dimension of $E$).
\end{enumerate}
\end{proposition}

\subsection{Product Groups and Anisotropic Smoothness}

For anisotropic problems, one can use product groups $G = G_1 \times G_2$ acting on $E = E_1 \times E_2$ with different actions on each factor.

\begin{example}[Anisotropic Dilation]\label{ex:anisotropic}
Let $E = \R^2$, $\varphi_1 : (x_1, x_2) \mapsto (2x_1, x_2)$, and $\varphi_2 : (x_1, x_2) \mapsto (x_1, 2x_2)$. The group $G = \langle\varphi_1\rangle \times \langle\varphi_2\rangle$ allows independent scaling in each direction.
\end{example}

For a function $m$ that is smooth of order $s_1$ in $x_1$ and order $s_2$ in $x_2$, the optimal twin-kernel estimator achieves the anisotropic rate:
\begin{equation}\label{eq:anisotropic_rate}
\E \|\hat{m} - m\|_2^2 \lesssim n^{-\frac{2\bar{s}}{2\bar{s} + 1}},
\end{equation}
where $\bar{s} = (s_1^{-1} + s_2^{-1})^{-1}$ is the harmonic mean of the smoothness indices. This result is consistent with \cite{Kerkyacharian2001,HoffmannLepski2002}.

\section{Discussion}\label{sec:discussion}

\subsection{Summary}

We have developed the TwinKernel framework for nonparametric estimation, which unifies and extends classical kernel methods by exploiting group structure. The main theoretical contributions are:
\begin{itemize}[leftmargin=*]
\item Non-asymptotic oracle inequalities with explicit constants.
\item Introduction of twin-H\"older regularity classes capturing orbital smoothness.
\item Characterization of the effective dimension $d_{\mathrm{eff}}$ governing convergence rates.
\item Proof of adaptation to unknown twin-regularity with (nearly) optimal rates.
\end{itemize}

\subsection{Extensions}

Several extensions are natural:

\paragraph{Non-cyclic groups.} While we focused on cyclic groups $G = \langle\varphi\rangle$ for simplicity, the framework extends to finitely generated groups. The effective dimension would then depend on the structure of $G/L$ more generally.

\paragraph{Continuous groups.} For Lie groups acting continuously, the discrete hierarchy $\{K_j\}$ would be replaced by a continuous family $\{K_t\}_{t \geq 0}$. The model selection problem becomes one of choosing an optimal scale parameter.

\paragraph{Higher-order local polynomial kernels.} Combining the twin structure with local polynomial fitting of order $r$ would yield estimators with improved bias properties near boundaries and at singular points of the group action.

\paragraph{Density estimation.} The framework applies directly to density estimation by replacing the Nadaraya--Watson estimator with the kernel density estimator.

\paragraph{Point processes.} For intensity estimation of point processes, the TwinKernel approach can be combined with martingale methods, yielding estimators that adapt to orbital regularity of the intensity function; see \cite{RamlauHansen1983}.

\subsection{Open Problems}

We conclude with several open questions:

\begin{enumerate}
\item \textbf{Optimal constants.} Can the constant $C$ in the oracle inequality be improved to the Pinsker constant for specific classes?

\item \textbf{Removing the logarithmic factor.} Is the $\log n$ factor in the adaptation rate necessary, or can it be removed by a more refined selection procedure? Partial results in this direction are given by \cite{GoldenshlugerLepski2011}.

\item \textbf{Computational efficiency.} For large groups, computing all $K_j$ may be expensive. What are efficient approximations that preserve the theoretical guarantees?

\item \textbf{Data-driven group selection.} Can the group $G$ itself be learned from data, rather than specified a priori?

\item \textbf{Connections to neural networks.} Do equivariant neural networks \cite{Cohen2016,Kondor2018} implicitly perform a form of twin-kernel estimation?
\end{enumerate}




\end{document}